\documentclass[twoside,11pt,reqno]{amsart}
\usepackage{amsmath,amssymb,amscd,mathrsfs,amscd}
\usepackage{graphics,verbatim}
\usepackage{enumitem}
\usepackage{hyperref}
\usepackage{graphicx}

\makeatletter
\newcommand*\bigcdot{\mathpalette\bigcdot@{.5}}
\newcommand*\bigcdot@[2]{\mathbin{\vcenter{\hbox{\scalebox{#2}{$\m@th#1\bullet$}}}}}
\makeatother

\newcommand{\Hom}{\ensuremath{\operatorname{Hom}}}

\newcommand{\Ind}{\ensuremath{\operatorname{ind}}}

\newcommand{\Dist}[1]{\operatorname{Dist}(#1)}

\newcommand{\U}[1]{\mathcal{U}_{#1}}
\newcommand{\Usmall}[1]{u_{#1}}

\newcommand{\St}{\operatorname{St}}

\makeatletter
\newcommand{\leqnomode}{\tagsleft@true}
\newcommand{\reqnomode}{\tagsleft@false}
\makeatother

\newtheorem{theorem}{Theorem}[subsection]

\makeatletter\let\c@fact\c@theorem\makeatother

\makeatletter\let\c@note\c@theorem\makeatother

\newtheorem{lemma}{Lemma}[subsection]
\makeatletter\let\c@lemma\c@theorem\makeatother

\makeatletter\let\c@lemma\c@theorem\makeatother

\makeatletter\let\c@alg\c@theorem\makeatother

\newtheorem{prop}{Proposition}[subsection]
\makeatletter\let\c@prop\c@theorem\makeatother

\makeatletter\let\c@conj\c@theorem\makeatother

\newtheorem{cor}{Corollary}[subsection]
\makeatletter\let\c@cor\c@theorem\makeatother

\makeatletter\let\c@defn\c@theorem\makeatother

\theoremstyle{definition}
\newtheorem{algorithm}{Algorithm}[subsection]

\newtheorem{remark}{Remark}[subsection]
\makeatletter\let\c@remark\c@theorem\makeatother

\makeatletter\let\c@example\c@theorem\makeatother
\numberwithin{equation}{subsection}

\usepackage[capitalise]{cleveref}
\crefname{theorem}{Theorem}{Theorems}
\crefname{fact}{Fact}{Facts}
\crefname{note}{Note}{Notes}
\crefname{lemma}{Lemma}{Lemmas}
\crefname{alg}{Algorithm}{Algorithms}
\crefname{remark}{Remark}{Remarks}
\crefname{example}{Example}{Examples}
\crefname{prop}{Proposition}{Propositions}
\crefname{conj}{Conjecture}{Conjectures}
\crefname{cor}{Corollary}{Corollaries}
\crefname{defn}{Definition}{Definitions}
\crefname{equation}{\!\!}{\!\!} 

\newcounter{listequation}

\begin{document}

\title{Steinberg quotients, Weyl Characters, and Kazhdan-Lusztig Polynomials}

\author{Paul Sobaje}
\address{Department of Mathematics \\
          Georgia Southern University}
\email{psobaje@georgiasouthern.edu}
\date{\today}

\begin{abstract}
Let $G$ be a reductive group over a field of prime characteristic.  An indecomposable tilting module for $G$ whose highest weight lies above the Steinberg weight has a character that is divisible by the Steinberg character.  The resulting ``Steinberg quotient" carries important information about $G$-modules, and in previous work we studied patterns in the weight multiplicities of these characters.  In this paper we broaden our scope to include quantum Steinberg quotients, and show how the multiplicities in these characters relate to algebraic Steinberg quotients, Weyl characters, and evaluations of Kazhdan-Lusztig polynomials.  We give an explicit algorithm for computing minimal characters that possess a key attribute of Steinberg quotients.   We provide computations which show that these minimal characters are not always equal to quantum Steinberg quotients, but are close in several nontrivial cases.
\end{abstract}

\maketitle


\section{Introduction}

\subsection{Overview}
This is a sequel to \cite{S1} in which we investigated characters of certain tilting modules.  In short, if $G$ is a reductive group in prime characteristic $p>0$, then an indecomposable tilting module for $G$ of the form $T((p-1)\rho +\lambda)$, where $\lambda$ is a $p$-restricted dominant weight, has a character that is divisible by the Steinberg character $\chi((p-1)\rho)$.  The resulting ``Steinberg quotient" $t(\lambda)$ is a nonnegative linear combination of $W$-orbit sums.  Thanks to the linkage principle, we can list which orbit sums might appear in $t(\lambda)$.  In previous work we proved that all such orbit sums do appear, and that their coefficients are weakly increasing in size as one moves down from the highest weight under the $\uparrow$ partial ordering.

In this paper we enlarge our investigation to consider $t(\lambda)$ for all dominant weights $\lambda$, as well as quantum Steinberg quotients $t_{\zeta}(\lambda)$, defined in the analogous way for tilting modules of a quantum group at a $p$-th root of unity.  In addition to the above pattern holding more generally, the wider scope makes clearer the connections between Steinberg quotients and more commonly studied quantities such as Weyl characters and Kazhdan-Lusztig combinatorics.  We will detail all of this below.

\subsection{Relationship to other tilting formulas}

Before stating our main results, let us comment briefly on the overlap between the topic of this paper and some existing results in the literature. Thanks to formulations by Soergel for quantum groups \cite{Soe}, and by Riche-Williamson for algebraic groups (stated in \cite{RW1}, and proved or re-proved in various contexts in \cite{AMRW} \cite{RW2} \cite{RW3} \cite{BR}), combinatorial algorithms for tilting characters are already known.  Moreover, in the case of quantum groups, the Steinberg quotients $t_{\zeta}(\lambda)$ are governed by the simple characters, and when $p>h$ the latter are given by Lusztig's Character Formula (LCF) from \cite{L} (see \cite[II.H.12]{rags} for an account of this).  In the algebraic setting, the analogous statement is not always true as it requires Donkin's tilting module conjecture to hold (it does not in general \cite{BNPS1} \cite{BNPS3}), we do not know precisely when the LCF describes the simple characters (\cite{AJS}, \cite{Fie2} \cite{W}), and in any case it would not apply to most $\lambda$ that are not $p$-restricted.

The main thrust of this work is to provide a complementary approach to computing tilting characters that applies only to special tilting modules, and exploits all of the unique properties that these modules possess.  The hope is that this can help answer questions that have not yet been answered by existing methods, such as an explanation as to when and why the characters $t(\lambda)$ and $t_{\zeta}(\lambda)$ differ.  Influences on the approach begun in \cite{S1} were Donkin's use of Brauer's formula in \cite[Proposition 5.5]{HTT}, along with work by Ye \cite{Ye} and Doty-Sullivan \cite{DS}.  In this sequel, we push the limits of these methods, while benefiting from the information and direction that the tilting character formulas mentioned above provide.

\subsection{Results and Organization}

Let $\mathbb{X}$ denote the character group of a maximal torus $T$ of $G$, and $\mathbb{Z}[\mathbb{X}]^W$ be the ring of $W$-invariants, where $W$ is the Weyl group of $G$.  The fact that the orbit sums in $t(\lambda)$ appear with weakly increasing multiplicity (when moving from the top orbit down) is due entirely to the fact that for all dominant weights $\mu$, the character product
\begin{equation}\label{E:multiplicationequation}
\chi((p-1)\rho+p\mu)t(\lambda)
\end{equation}
has nonnegative coefficients when expressed in the Weyl character basis.  Though this argument is present in \cite{S1}, its importance is more explicitly isolated here in Theorem \ref{T:MainGen}, where we give a broader statement that highlights the similarity between the orbit-sum multiplicities in Steinberg quotients and those in Weyl characters (a further parallel will be noted shortly).

With this theorem in hand, the extension of the main result from \cite{S1} to the Steinberg quotients $t(\lambda)$ and $t_{\zeta}(\lambda)$, for any dominant $\lambda$, follows from well-known facts about tilting modules.  We also record other features of these characters that, though easy to prove, give interesting perspective.  For example, we obtain a natural framework in which Steinberg quotients become an enlargement of sorts to the set of Weyl characters.  That is, for $\lambda$ dominant, the quotient $t_{\zeta}(p\lambda)=\chi(\lambda)^F$, where $F$ is the Frobenius twist on a character.

In Section 4 we give direct comparisons between algebraic and quantum Steinberg quotients.   From what is already known about the relationship between tilting modules in the respective categories, it follows that the characters $t(\lambda)$ can be written as nonnegative sums of the various $t_{\zeta}(\mu)$.  By using base-changing results from \cite{Lin}, \cite{PS}, and \cite{And2}, we give more precise statements on this relationship.  One interesting consequence is that for the Steinberg quotients $q(\lambda)$ of the $G_1T$-indecomposable modules, we can show (under a minor condition on $p$) that all possible orbits appear with positive multiplicity, even when $q(\lambda)\neq t(\lambda)$.  This could be viewed as an analog in this setting to the Premet-Suprunenko theorem on the weight sets of the $p$-restricted simple $G$-modules \cite{Pr} \cite{Su}.

Suppose now that $p \ge h$, where $h$ is the Coxeter number of the underlying root system, and assume that $\lambda-\rho$ is a $p$-regular weight.  Applying work by Kato \cite{K}, we show in Section 5 that when the LCF describes the simple characters (in the respective settings), then the orbit multiplicities in $t_{\zeta}(\lambda)$ are given by evaluations of Kazhdan-Lusztig polynomials, and the same is true for $q(\lambda)$ when $\lambda$ is $p$-restricted.  The hypothesis does hold in the quantum setting when $p>h$, but will not hold in general in the algebraic setting unless $p \gg h$.  We should also point out that when this condition holds in both settings, then there is an agreement $t_{\zeta}(\lambda)=q(\lambda)$ for all $p$-restricted weights $\lambda$ (i.e. including the $p$-singular ones).  Of course we also have $t_{\zeta}(\lambda)=t(\lambda)$ provided that $q(\lambda)=t(\lambda)$ (this last equality always holding when $p \ge 2h-4$).  In making the connection to Kazhdan-Lusztig polynomials, the heavy lifting is done by Kato's paper along with Fiebig's detailed account of it \cite{Fie}.

Our ultimate goal is to find character formulas for Steinberg quotients that can, at minimum, differentiate between $t_{\zeta}(\lambda)$ and $t(\lambda)$, and that will lend themselves to reasonable dimension formulas (akin to $p$-versions of Weyl's dimension formula).  In order to achieve this, it is necessary to find the defining properties of Steinberg quotients.  We initiate this investigation in Section 6.

In view of the results in Section 3, we begin by defining the character $\mathcal{M}_p(\lambda)$ to be the smallest element in $\mathbb{Z}[\mathbb{X}]^W$ that satisfies (\ref{E:multiplicationequation}) and has $\lambda$ as its highest weight.  In Theorem \ref{T:Thisisenough} we prove that this property can be checked by multiplying against a finite number of characters of the form $\chi((p-1)\rho+p\mu)$, though in general checking against $\chi((p-1)\rho)$ alone will not be sufficient.  We then give an explicit a process for computing $\mathcal{M}_p(\lambda)$ in Algorithm \ref{Algorithm1}.

Since the $t_{\zeta}(\lambda)$ are lower bounds on the $t(\lambda)$, and can be computed by ordinary Kazhdan-Lusztig polynomials when $p>h$, it is both natural and possible to check to see how close these are to the $\mathcal{M}_p(\lambda)$.  Surprisingly, in all of the computations that we were able to make, they were very close.  They were equal for all restricted $\lambda$ with $\lambda-\rho$ a $p$-regular weight for root systems $A_1,A_2,A_3$ (the first two being trivial), and for almost all such weights in type $A_4$, and for many large weights in type $A_5$.  The cases of character equality are nontrivial, with orbit multiplicities as large as $23$ occurring, and in the few cases we found in which they were not equal, it was by the smallest margin possible (a multiplicity difference of $1$ on the lowest orbits).  In many of these cases we can also compute the characters $t(\lambda)$, thanks to knowing that $t(\lambda)=q(\lambda)$ from \cite{BNPS2} and \cite{BNPS4}, and that the LCF describes the $p$-restricted simple characters from \cite[II.8.22]{rags} and \cite{Sc}.

\subsection{Acknowledgements}

We thank Frank L\"{u}beck for generously sharing the extensive computations of Kazhdan-Lusztig polynomials made by the algorithms described in \cite{Lu}.

\section{Notation and Recollections}

\subsection{Weyl groups, Roots, and Weights}
We give a brief overview on our notation.  For the most part it follows \cite{rags}, and any notation not explicitly mentioned may be assumed to be consistent with that.

Let $\Bbbk$ be an algebraically closed field of characteristic $p>0$.  By standard arguments we may consider $G$ to be a simple and simply connected group, the results for which can be generalized to any $G$ connected reductive.

Fix a maximal torus $T$ inside a Borel subgroup $B$ of $G$.  The root system is denoted $\Phi$, and we fix a set of simple roots $S=\{\alpha_1,\alpha_2,\ldots,\alpha_n\}$, where $n$ is the rank of $T$.  This determines a set of positive roots $\Phi^+ \subseteq \Phi$.  Denote by $\mathbb{X}$ the character group of $T$ (also called the set of weights).  Each $\alpha \in \Phi^+$ has a corresponding coroot $\alpha^{\vee}$.  The highest short root is $\alpha_0$, and $\alpha_0^{\vee}$ is the highest coroot.  For each $\lambda \in \mathbb{X}$ and coroot $\alpha^{\vee}$ we denote the natural pairing by $\langle \lambda, \alpha^{\vee} \rangle$.  The set of dominant weights is $\mathbb{X}^+$, and it generated over $\mathbb{Z}_{\ge 0}$ by the fundamental dominant weights $\{\varpi_1, \varpi_2, \ldots, \varpi_n\}$, which are defined by the property that $\langle \varpi_i, \alpha_j^{\vee} \rangle = \delta_{ij}$.
For each $m \ge 0$, we define
$$\mathbb{X}_m = \{a_1\varpi_1+\cdots a_n\varpi_n \mid 0 \le a_i < m\} \subseteq \mathbb{X}^+.$$
Thus $\mathbb{X}_p$ denotes the $p$-restricted dominant weights (we note that we have often just used $\mathbb{X}_p$ for this set in the past, but require the finer notation in this paper).

The root lattice is $\mathbb{Z}\Phi \subseteq \mathbb{X}$.  The element $\rho$ is the half-sum of the positive roots, or equivalently is the sum of the fundamental dominant weights.  The Weyl group is $W$, and $w_0$ is its longest element.  For any $\lambda \in \mathbb{X}$, we let $W_{\lambda}$ denote the stabilizer of $\lambda$, while $W\lambda$ is the $W$-orbit of $\lambda$.

The standard partial order on $\mathbb{X}$ is denoted as $\le$.  The affine Weyl group is
$$W_p \cong W \ltimes p\mathbb{Z}\Phi,$$
and it acts on $\mathbb{X}$.  It can be shown that the image of $W_p$ in the group of affine transformtions of $\mathbb{E}= \mathbb{R} \otimes_{\mathbb{Z}} \mathbb{X}$ is generated by affine reflections of the form
$$s_{\alpha,np}(\lambda)=\lambda - (\langle \lambda, \alpha^{\vee} \rangle - np)\alpha$$ for all $\alpha \in \Phi^+$ and $n \in \mathbb{Z}$. 
For each $w \in W_p$ and $\lambda \in \mathbb{E}$, we denote the action of $w$ on $\lambda$ by juxtaposition, as $w\lambda$.  We will primary by interested in the ``dot action" of $W_p$, where
$$w \bigcdot \lambda = w(\lambda + \rho) - \rho.$$

For each $\alpha \in \Phi^+, n \in \mathbb{Z}$, there is a hyperplane in $\mathbb{E}$ defined by
$$H_{\alpha,np}= \{ \lambda \in \mathbb{E} \mid \langle \lambda+\rho, \alpha^{\vee} \rangle = np \}.$$
The affine reflection of $\mathbb{E}$ about $H_{\alpha,np}$ is precisely the dot action of $s_{\alpha,np}$ on $\mathbb{E}$.  The partial ordering $\uparrow$ on $\mathbb{X}$ is the minimal such ordering with the property that
$$(s_{\alpha,np} \bigcdot \lambda) \uparrow \lambda \quad \text{if} \quad (s_{\alpha,np} \bigcdot \lambda) \le \lambda,$$
and
$$\lambda \uparrow (s_{\alpha,np} \bigcdot \lambda) \quad \text{if} \quad \lambda \le (s_{\alpha,np} \bigcdot \lambda).$$
More generally, $\lambda \uparrow \mu$ if there are affine reflections $s_1,\ldots,s_m$ such that
\begin{equation}\label{E:alcovestring}
\lambda \le s_1 \bigcdot \lambda \le s_2 \bigcdot s_1 \bigcdot \lambda \le \cdots \le s_m \bigcdot \cdots \bigcdot s_1 \lambda = \mu.
\end{equation}
Properties of this ordering are noted in \cite[II.6.4]{rags}.

The hyperplanes $H_{\alpha,np}$ divide $\mathbb{E}$ up into a system of alcoves and facets.  The alcoves contain points from $\mathbb{X}$ if and only if $p \ge h$.  The elements in the alcoves are called $p$-regular weights.  They are those weights $\lambda$ such that
$$\langle \lambda + \rho, \alpha^{\vee} \rangle \not\in p\mathbb{Z}$$
for all $\alpha \in \Phi^+$.

Let $\mathcal{C}$ denote the set of all alcoves of $\mathbb{E}$.  The lowest dominant alcove $C_0$ is the alcove
$$C_0 = \{ \lambda \in \mathbb{E} \mid 0< \langle \lambda+\rho, \alpha^{\vee} \rangle < p\}.$$
The action of $W_p$ on $\mathcal{C}$ is simply transitive, hence for any alcove $C \in \mathcal{C}$ there is a unique element $w \in W_p$ such that $w.C_0=C$.

An alcove $C$ is called dominant if $0< \langle \lambda+\rho, \alpha_i^{\vee} \rangle$ for all $i$, and an alcove is $p$-restricted if $0< \langle \lambda+\rho, \alpha_i^{\vee} \rangle < p$ for all $i$.

The group $W_p$ is a Coxeter group with generators $\{s_0,s_1,\ldots,s_n\}$, where for $1 \le i \le n$ we have $s_i = s_{\alpha_i,0}$, and $s_0=s_{\alpha_0,p}$.  These generators are just the affine reflections about the hyperplanes that extend the $n+1$ walls of the fundamental alcove.

\subsection{Characters}

The Grothendieck ring of the category of finite dimensional $T$-modules is isomorphic to the group algebra $\mathbb{Z}[\mathbb{X}]$.  For each $\mu \in \mathbb{X}$, we denote by $e(\mu)$ the corresponding basis element in $\mathbb{Z}[\mathbb{X}]$.  Since $\mathbb{Z}[\mathbb{X}]$ is the group algebra of a free abelian group of rank $n$, it is isomorphic to the ring of Laurent polynomials over $\mathbb{Z}$ in $n$ indeterminants.  In particular, $\mathbb{Z}[\mathbb{X}]$ is an integral domain, so the cancellation property for ring multiplication holds.

We denote by $s(\mu)$ the sum of the weights in the $W$-orbit of $\mu$.  These elements form a basis of $\mathbb{Z}[\mathbb{X}]^W$.

Recall that for $\sigma = \sum a_{\mu}e(\mu) \in \mathbb{Z}[\mathbb{X}]$, is ``dual" and ``Frobenius twist" are
$$\sigma^*=\sum a_{\mu}e(-\mu),$$
and
$$\sigma^F=\sum a_{\mu}e(p\mu)$$
respectively.  If $\sigma = \text{ch}(M)$ for a $T$-module $M$, then $\sigma^* = \text{ch}(M^*)$, and $\sigma^F = \text{ch}(M^{(1)})$.

\subsection{$G$-modules and $G_1T$-modules}
For each $\lambda \in \mathbb{X}^+$ there is a simple $G$-module $L(\lambda)$, a costandard module $\nabla(\lambda)=\Ind_B^G \lambda$, a standard module $\Delta(\lambda)=(\Ind_B^G -w_0\lambda)^*$, and an indecomposable tilting module $T(\lambda)$.  The modules $\Delta(\lambda)$ and $\nabla(\lambda)$ each have character given by the Euler characteristic
$$\chi(\lambda)= \sum_{i \ge 0}(-1)^i (\text{ch} \, R^i \Ind_B^G \lambda).$$

By the strong linkage principle \cite[II.6.13]{rags}, $[\nabla(\lambda):L(\mu)] > 0$ implies that $\mu \uparrow \lambda$.  In a similar way, write $(T(\lambda):\chi(\mu))$ for the multiplicity of $\nabla(\mu)$ in a good filtration of $T(\lambda)$ (equal to the multiplicity of $\Delta(\mu)$ in a Weyl filtration of $T(\lambda)$.  If $(T(\lambda):\chi(\mu)) > 0$, then again $\mu \uparrow \lambda$ \cite[II.E.3]{rags}. 

For each $\lambda \in X$ there is a simple $G_1T$-module $\widehat{L}_1(\lambda)$, a projective indecomposable $G_1T$-module $\widehat{Q}_1(\lambda)$, and ``baby Verma modules"
$$\widehat{Z}_1(\lambda) = \textup{coind}_{B_1^+T}^{G_1T} \lambda, \qquad \widehat{Z}_1^{\prime}(\lambda) = \textup{ind}_{B_1T}^{G_1T} \lambda.$$

Fix a Frobenius endomorphism $F:G \rightarrow G$.  For any $G$-module $M$, we denote by $M^{(1)}$ its twist under $F$.

\subsection{Quantum Groups}

Let $v$ be an indeterminate, and $\mathbb{Q}(v)$ the fraction field of $\mathbb{Q}[v]$.  The quantum group $\U{v}$ is the $\mathbb{Q}(v)$-algebra with generators $E_{\alpha}, F_{\alpha}, K_{\alpha}^{\pm 1}$, for $\alpha \in \Pi$, satisfying the quantum Serre relations of \cite[H.2]{rags}.  Over the subring $A=\mathbb{Z}[v,v^{-1}]$, we denote by $\U{A}$ Lusztig's divided power integral form for $\U{v}$.  The algebra $\U{A}$ is free as an $A$-module, and the multiplication map $\U{A} \otimes_A \mathbb{Q}(v) \rightarrow \U{v}$ is an isomorphism of rings.  

For any commutative $A$-algebra $B$ one obtains the quantum group $\U{B} = \U{A} \otimes_A B$.  Let $\zeta$ be a complex primitive $p$-th root of unity.  Specializing $v=\zeta$ makes $\mathbb{C}$ into an $A$-algebra.  We now denote by $\U{\zeta}$ the resulting quantum group $\U{A} \otimes_A \mathbb{C}$.

The category of finite dimensional $\U{\zeta}$-modules, denoted $\U{\zeta}$-mod, has many similarities to that of $G$-mod.  First, it is known that the category breaks into a direct sum of subcategories based on central characters, and we restrict our attention only to the subcategory of  type \textbf{1} $\U{\zeta}$-modules.  In this subcategory, for each $\lambda \in \mathbb{X}^+$ there is a simple module $L_{\zeta}(\lambda)$, a standard module $\Delta_{\zeta}(\lambda)$, a costandard module $\nabla_{\zeta}(\lambda)$, and an indecomposable tilting module $T_{\zeta}(\lambda)$.

As we are considering only type \textbf{1} modules, we will regard the quantum Frobenius morphism as a surjective homomorphism $F:\U{\zeta} \rightarrow \mathcal{U}(\mathfrak{g}_{\mathbb{C}})$ (note then that the image of $F$ as defined here is the quotient of the image of the more commonly defined quantum Frobenius morphism).  Let $L_{\mathbb{C}}(\lambda)$ denote the irreducible $\mathfrak{g}_{\mathbb{C}}$-module of highest weight $\lambda$.  The pullback under $F$ will be denoted $L_{\mathbb{C}}(\lambda)^F$.  If $\lambda=\lambda_0+p\lambda_1$ with $\lambda_0 \in \mathbb{X}_p$ and $\lambda_1 \in \mathbb{X}^+$, then there is an isomorphism of $\U{\zeta}$-modules
$$L_{\zeta}(\lambda) \cong L_{\zeta}(\lambda_0) \otimes L_{\mathbb{C}}(\lambda_1)^F.$$

The character of a type \textbf{1} finite dimensional $\U{\zeta}$-module is also an element in $\mathbb{Z}[\mathbb{X}]^W$.  We have
$$\chi(\lambda)=\text{ch}(\Delta_{\zeta}(\lambda)) = \text{ch}(\nabla_{\zeta}(\lambda)).$$

The small quantum group is denoted $\Usmall{\zeta}$.

\section{Steinberg Quotients}

\subsection{}\label{SS:SteinbergRecall}

In \cite{S1} we defined, for each $\lambda \in \mathbb{X}_p$, the Steinberg quotients
$$t(\lambda) = T((p-1)\rho+\lambda)/\chi((p-1)\rho)$$
and
$$q(\lambda) = \widehat{Q}_1((p-1)\rho+w_0\lambda)/\chi((p-1)\rho).$$

\vspace{.1in}
For $p \ge 2h-4$ these two characters are the same \cite{BNPS4}, though in general they can differ (see also \cite{BNPS1}, \cite{BNPS2}, and \cite{BNPS3}).  We note that \cite{BNPS4} actually utilizes the language of Steinberg quotients, and establishes nice properties about their restrictions to Levi subgroups.

There are non-negative integers $a_{\mu,\lambda}$ and $b_{\mu,\lambda}$ \footnote{In \cite{S1} we denoted the double index $a_{\mu,\lambda}$ as $a_{\mu}^{\lambda}$, and $b_{\mu,\lambda}$ as $b_{\mu}^{\lambda}$.} such that
$$q(\lambda) = \sum a_{\mu,\lambda} s(\mu)$$
and
$$t(\lambda) = \sum b_{\mu,\lambda} s(\mu).$$
Computing Steinberg quotients then amounts to determining these orbit multiplicities, and the Steinberg quotients in turn give the characters of the relevant modules upon multiplying by $\chi((p-1)\rho)$.

Some preliminary properties that can be established for these coefficients is that for all $\lambda \in \mathbb{X}_p$ and $\mu \in \mathbb{X}^+$:
\begin{enumerate}
    \item $a_{\lambda,\lambda}=b_{\lambda,\lambda}=1$
    \item $a_{\mu,\lambda}\le b_{\mu,\lambda}$
    \item $a_{\mu,\lambda} \ne 0$ implies $(\mu-\rho) \uparrow (\lambda-\rho)$.
    \item $a_{\mu,\lambda} = b_{\mu,\lambda}$ \quad if \quad $p \ge 2h-4$
\end{enumerate}

\bigskip
The main result proved in \cite{S1} was a multiplicity pattern in the coefficients $b_{\mu,\lambda}$.

\begin{theorem}\cite[Theorem 3.2.2]{S1}\label{T:OldMain}
Let $\lambda \in \mathbb{X}^+$, and let $b_{\mu,\lambda}$ be non-negative integers such that
$$t(\lambda) = \sum_{\mu \in \mathbb{X}^+} b_{\mu,\lambda}s(\mu).$$
For dominant weights $\mu,\mu^{\prime}$,
$$\textup{if} \quad (\mu-\rho) \uparrow (\mu^{\prime}-\rho)\textup{, \qquad then \quad} b_{\mu,\lambda} \ge b_{\mu^{\prime},\lambda}.$$
\end{theorem}

We will generalize this result in Theorem \ref{T:MainGen}, which distills the essential character arguments.  The original proof, and therefore this generalized one also, was inspired by ideas due to Ye \cite{Ye}, Doty and Sullivan \cite{DS}, and Donkin \cite[Proposition 5.5]{HTT}.

The statement give in  Theorem \ref{T:MainGen} will be about formal characters, but applies to Steinberg quotients thanks to the following fundamental results that hold in the category of $G$-modules.

\bigskip
\begin{itemize}
    \item For every $w \in W$,
    $$\chi(w \bigcdot \lambda)=(-1)^{\ell(w)}\chi(\lambda).$$
    \item Brauer's formula: for all $\lambda,\mu \in \mathbb{X}^+$,
    $$\chi(\lambda)s(\mu)=\sum_{w \in W/W_{\mu}} \chi(\lambda+w\mu).$$
    \item The Andersen-Haboush Theorem: for all $\mu \in \mathbb{X}^+$,
    $$\nabla((p-1)\rho+p\mu) \cong \St \otimes \nabla(\mu)^{(1)}.$$
    \item For all $\lambda, \mu \in \mathbb{X}^+$, the module
    $$T((p-1)\rho+\lambda) \otimes \nabla(\mu)^{(1)}$$
    has a good filtration.
\end{itemize}

\bigskip
The last result follows from the Andersen-Haboush Theorem together with the fact that the tensor product of good filtration modules has a good filtration (proved for certain $p$ by Wang, most $p$ by Donkin, and all $p$ by Mathieu).  One then observes that $T((p-1)\rho+\lambda) \otimes \nabla(\mu)^{(1)}$ is a direct summand of
$$\St \otimes T(\lambda) \otimes \nabla(\mu)^{(1)} \cong \nabla((p-1)\rho + p\mu) \otimes T(\lambda).$$

\subsection{}

In this subsection we collect a number of lemmas that will simplify proofs both in this section, and then later on in the paper.

\begin{lemma}\label{L:dotandregularaction}
For all $w \in W$ and $\lambda, \mu \in \mathbb{X}$,
$$w \bigcdot (\lambda + \mu) = w \bigcdot \lambda + w\mu.$$
\end{lemma}

\begin{proof}
We have
\begin{align*}
w \bigcdot (\lambda+\mu) & = w(\lambda+\mu+\rho)-\rho\\
& = w(\lambda+\rho) -\rho + w\mu\\
& = w \bigcdot \lambda + w\mu.\\
\end{align*}

\end{proof}

\begin{lemma}\label{L:StraighteningInterior}
Let $\lambda \in \mathbb{X}$, and $\gamma \in \mathbb{X}^+$.  If for some $\alpha \in \Phi^+$,
$$\langle \lambda+\gamma+\rho,\alpha^{\vee} \rangle < 0,$$
then
$$s_{\alpha} \bigcdot (\lambda+\gamma) - \gamma$$
lies strictly between $\lambda$ and $s_{\alpha}\lambda$.  In particular, this weight is in the interior of $\textup{conv}(W\lambda)$.
\end{lemma}

\begin{proof}
In this case, since
$$\langle \gamma+\rho,\alpha^{\vee} \rangle > 0,$$
we must have that
$$0 > \langle \lambda + \gamma+\rho,\alpha^{\vee} \rangle > \langle \lambda,\alpha^{\vee} \rangle.$$
Therefore
\begin{align*}
s_{\alpha} \bigcdot (\lambda+\gamma) - \gamma & = s_{\alpha}(\lambda+\gamma+\rho)-\rho-\gamma\\
& = \lambda - \langle \lambda + \gamma+\rho,\alpha^{\vee} \rangle \alpha\\
& < \lambda - \langle \lambda,\alpha^{\vee} \rangle \alpha\\
& = s_{\alpha}\lambda.
\end{align*}
\end{proof}

\begin{lemma}\label{L:NeededReflections}
Let $\lambda, \gamma \in \mathbb{X}^+$.  Let $J \subset \Pi$ be the set of all simple roots $\alpha_i$ such that
$$\langle \gamma+\rho,\alpha_i^{\vee} \rangle \le \langle \lambda, \alpha_0^{\vee} \rangle.$$
Then for any $w_0\lambda \le \mu \le \lambda$, there is a $w \in W_J$ such that
$$w \bigcdot (\gamma+\mu) \in \mathbb{X}^+ -\rho.$$
\end{lemma}

\begin{proof}
For all $\alpha_i \in \Pi$, the bounding on $\mu$ implies that
$$\langle w_0\lambda, \alpha_0^{\vee} \rangle \le \langle \mu, \alpha_i^{\vee} \rangle \le \langle \lambda, \alpha_0^{\vee} \rangle.$$
Therefore, if
\begin{align*}
0 & > \langle (\gamma+\mu)+\rho,\alpha_i^{\vee} \rangle\\
 & = \langle \gamma+\rho,\alpha_i^{\vee} \rangle + \langle \mu, \alpha_i^{\vee} \rangle\\
& \ge \langle \gamma+\rho,\alpha_i^{\vee} \rangle - \langle \lambda, \alpha_0^{\vee} \rangle,
\end{align*}
then it follows that $\alpha_i \in J$.  We may replace $\gamma + \mu$ with $s_i \bigcdot (\gamma+\mu)$, which is on the positive side of the hyperplane $H_{\alpha_i,0}$.  It follows by the previous lemma, and our assumption on $\mu$, that
$$s_i \bigcdot (\gamma+\mu)=\gamma+\mu^{\prime},$$
with $w_o\lambda \le \mu^{\prime} \le \lambda$.  We may now repeat the process above with $\gamma + \mu^{\prime}$, and will eventually wind up with a weight in $\mathbb{X}^+ - \rho$ having only used reflections from $W_J$.

\end{proof}

\subsection{}

Recall that Weyl characters refer to those Euler characteristics $\chi(\lambda)$ for which $\lambda \in \mathbb{X}^+$.  The Weyl characters form a $\mathbb{Z}$-basis for $\mathbb{Z}[\mathbb{X}]^W$.  A nonzero element in $\mathbb{Z}[\mathbb{X}]^W$ having nonnegative coefficients in the Weyl basis will be called a \textit{good filtration character}.

\begin{samepage}
\begin{theorem}\label{T:MainGen}
Let $\eta \in \mathbb{Z}[\mathbb{X}]^W$, where $\eta = \sum_{\mu \in \mathbb{X}^+} c_{\mu} s(\mu)$.  
\begin{enumerate}
\item Suppose that for every $\lambda \in \mathbb{X}^+$, the product $\chi(\lambda)  \eta$ is a good filtration character.  Then $c_{\mu} \ge c_{\mu^{\prime}}$ whenever $\mu \le \mu^{\prime}$.
\\
\item Suppose that for every $\lambda \in \mathbb{X}^+$, the product $\chi((p-1)\rho+p\lambda)  \eta$ is a good filtration character.  Then $c_{\mu} \ge c_{\mu^{\prime}}$ whenever $\mu-\rho \uparrow \mu^{\prime}-\rho$.
\end{enumerate}
\end{theorem}
\end{samepage}

\begin{proof}
(1) It suffices to prove the result in the case $\mu^{\prime}$ is a minimal dominant weight such that $\mu^{\prime} > \mu$.  Under this assumption, \cite[Theorem 2.6]{St} shows that there is a positive root $\alpha \in \Phi^+$ such that $\mu + \alpha = \mu^{\prime}$.
Set
$$n= \langle \mu, \alpha^{\vee} \rangle.$$
We then have that $ \langle \mu^{\prime}, \alpha^{\vee} \rangle = n+2$, and since $\mu$ is dominant, that $n \ge 0$.  There is a simple root $\alpha_i$ and an element $w \in W$ such that $w\alpha=-\alpha_i$.  From this it follows that
$$w\mu^{\prime}=w(\mu+\alpha)=w\mu-\alpha_i.$$
We also have
$$\langle w\mu, \alpha_i^{\vee} \rangle= -n,$$
and
$$ \langle w\mu^{\prime}, \alpha_i^{\vee} \rangle = -(n+2).$$
Set $$\gamma = \sum m_j \varpi_j \in \mathbb{X}^+$$
where $m_i=n$, and for $j \ne i$,
$$m_j > \langle \sigma, \alpha_0^{\vee} \rangle,$$
for all weights $\sigma$ appearing in $\eta$ (note that such a choice is possible as there are only finitely many such $\sigma$).  By Brauer's formula,
\begin{equation}\label{E:BrauerSum}
\chi(\gamma)\eta =\sum_{\lambda \in \mathbb{X}^+} \sum_{\sigma \in W\lambda} c_{\lambda} \chi(\gamma+\sigma).
\end{equation}

Among the terms of this sum are $c_{\mu}\chi(\gamma+w\mu)$ and $c_{\mu^{\prime}}\chi(\gamma+w\mu^{\prime})$.  For any $\sigma$ that is a weight in $\eta$, it follows from Lemma \ref{L:NeededReflections} that the choice of $\gamma$ guarantees that either
$$\sigma+\gamma \in \mathbb{X}^+-\rho,$$
or else
$$s_i \bigcdot (\sigma + \gamma) \in \mathbb{X}^+ - \rho.$$
We see in particular that $\gamma+w\mu \in \mathbb{X}^+-\rho$, and in fact is in $\mathbb{X}^+$, and that
\begin{align*}
s_i \bigcdot (\gamma+w\mu^{\prime}) & = s_i(\gamma+w\mu^{\prime}+\rho)-\rho\\
& = \gamma+w\mu^{\prime}+\rho - \langle \gamma+w\mu^{\prime}+\rho, \alpha_i^{\vee} \rangle \alpha_i -\rho\\
& = \gamma+w\mu^{\prime}+\rho + \alpha_i -\rho\\
& = \gamma+w\mu.\\
\end{align*}
It then follows that when the sum in (\ref{E:BrauerSum}) is rewritten as a sum of Weyl characters, the coefficient on $\chi(\gamma+w\mu)$ is $c_{\mu}-c_{\mu^{\prime}}$.  By hypothesis, this coefficient is nonnegative, proving that $c_{\mu} \ge c_{\mu^{\prime}}$.

(2) This case follows a similar logic, though there are a few modifications that need to be spelled out.  First, we assume that the relation $\mu-\rho \uparrow \mu^{\prime} - \rho$ is minimal, so that there is some $\alpha \in \Phi^+$ and some $n \ge 1$ such that
$$s_{\alpha,np} \bigcdot (\mu-\rho)=\mu^{\prime}-\rho.$$
This implies that
$$\langle \mu-\rho+\rho, \alpha^{\vee} \rangle < np < \langle \mu^{\prime}-\rho+\rho, \alpha^{\vee} \rangle.$$
Equivalently,
$$\langle \mu, \alpha^{\vee} \rangle < np < \langle \mu^{\prime}, \alpha^{\vee} \rangle$$
Again there is a simple root $\alpha_i$ and an element $w \in W$ such that $w\alpha=-\alpha_i$.
We then have that
$$\langle w\mu, \alpha_i^{\vee} \rangle < -np < \langle w\mu^{\prime}, \alpha_i^{\vee} \rangle.$$
We now set $$\gamma = \sum m_j \varpi_j \in \mathbb{X}^+$$
where $m_i=n-1$, and for $j \ne i$, 
$$p(m_j+1) > \langle \sigma, \alpha_0^{\vee} \rangle.$$
The proof now follows similar concluding logic as in proof of (1).  The character
$$\chi((p-1)\rho+p\gamma)\eta$$
has by assumption nonnegative coefficients when expressed in the Weyl basis, and one can verify that $$c_{\mu}-c_{\mu^{\prime}}$$
is one of these coefficients.
\end{proof}

\begin{remark}
In the first statement of this theorem, the fact that $\chi(0)\eta$ is a good filtration character means that $\eta$ is a good filtration character.  In this case the theorem is simply giving a statement about orbit sums in Weyl characters, and this property is already known.  We put these together so that Weyl characters and Steinberg quotients can be seen as parallel in a certain sense.
\end{remark}

\subsection{}
In \cite{S1}, the quotients $t(\lambda)$ and $q(\lambda)$ were defined only for $\lambda \in \mathbb{X}_p$.  While there is not much point in extending the definition of the $q(\lambda)$ to include more weights (one could make sense of such a definition, but we effectively obtain nothing new as follows from \cite[II.11.3(2)]{rags}), extending the definition of $t(\lambda)$ turns out to be quite useful.  That is, for $\lambda \in \mathbb{X}^+$ we define (as before)

$$t(\lambda)=\textup{ch}(T((p-1)\rho+\lambda))/\chi((p-1)\rho).$$

\vspace{0.2in}

\noindent The facts recalled in Section \ref{SS:SteinbergRecall} are true of $T((p-1)\rho+\lambda)$ for all dominant $\lambda$.  Therefore we may apply Theorem \ref{T:MainGen}(2) to $t(\lambda)$, showing that the statement in Theorem \ref{T:OldMain} holds in this setting also.

By extending this definition, we can calculate precisely a number of the $t(\lambda)$.  This next result follows immediately from Donkin's tensor product theorem \cite[Proposition 2.1]{D3}.

\begin{prop}
Let $\lambda \in \mathbb{X}_p$.  If $t(\lambda)=q(\lambda)$, then
$$t(\lambda+p\mu)=t(\lambda)\textup{ch}(T(\mu))^F.$$
\end{prop}

\subsection{}

The quantum Steinberg module is $\St_{\zeta} = L_{\zeta}((p-1)\rho)$.  Let $\mu \in X_+$.  The character equality
$$\chi((p-1)\rho+p\mu)=\chi((p-1)\rho)\chi(\mu)^F$$
reflects the module isomorphism
$$\nabla_{\zeta}((p-1)\rho+p\mu) \cong \St_{\zeta} \otimes L_{\mathbb{C}}(\mu)^F.$$
The module $\nabla_{\zeta}((p-1)\rho+p\mu)$ is simultaneously simple, standard, costandard, and tilting.  For each $\lambda \in \mathbb{X}^+$, the tilting module $T_{\zeta}((p-1)\rho+\lambda)$ is an injective and projective $\U{\zeta}$-module, and every an indecomposable injective $\U{\zeta}$-module is of this form.  Writing $\lambda = \lambda_0 + p\lambda_1$, with $\lambda_0 \in \mathbb{X}_p$ and $\lambda_1 \in X_+$, we have
$$T_{\zeta}((p-1)\rho+\lambda) \cong T_{\zeta}((p-1)\rho+\lambda_0) \otimes L_{\mathbb{C}}(\lambda_1)^F.$$

\vspace{.2in}

Let $\lambda \in \mathbb{X}^+$.  We define the quantum Steinberg quotient $t_{\zeta}(\lambda)$ by
$$t_{\zeta}(\lambda) = \text{ch}(T_{\zeta}((p-1)\rho+\lambda))/\chi((p-1)\rho).$$
Since $\text{ch}(T_{\zeta}((p-1)\rho+\lambda))$ is $W$-invariant, there are non-negative integers $c_{\mu,\lambda}$ such that
$$t_{\zeta}(\lambda) = \sum c_{\mu,\lambda} s(\mu).$$

\begin{theorem}
The statement of Theorem \ref{T:OldMain} for the coefficients $b_{\mu,\lambda}$ holds also for the coefficients $c_{\mu,\lambda}$.
\end{theorem}

Again, we may apply Theorem \ref{T:MainGen}(2) to see that the statement of Theorem \ref{T:OldMain} for the coefficients $b_{\mu,\lambda}$ holds also for the $c_{\mu,\lambda}$.

\begin{prop}\label{P:quantumtiltingtensor}
For each $\lambda \in \mathbb{X}_p$ and $\mu \in \mathbb{X}^+$ we have
$$t_{\zeta}(\lambda+p\mu)=t_{\zeta}(\lambda)\chi(\mu)^F.$$
\end{prop}

\section{Comparison between algebraic and quantum Steinberg quotients}

\subsection{}
We will primarily follow the $p$-modular setup of \cite{PS}, but also refer the interested reader to \cite{Lin}, where some of the modules below were first studied.  Recall that $A=\mathbb{Z}[v,v^{-1}]$.  As above, $\zeta$ denotes a fixed complex primitive $p$-th root of unity.  Let $\mathcal{O}$ denote the local ring $\mathbb{Z}[\zeta]_{(1-\zeta)}$.  The assignment $v \mapsto \zeta$ defines a ring homomorphism $A \rightarrow \mathcal{O}$.  Set
$$\U{\mathcal{O}} = \U{A} \otimes_A \mathcal{O}.$$

Then $\U{\mathcal{O}}$ is also a kind of integral form for $\U{\zeta}$.  Since $\mathcal{O}$ has residue field $\mathbb{F}_p$, we have $\U{\Bbbk} \cong \U{\mathcal{O}} \otimes_{\mathcal{O}} \Bbbk$.  There is a surjective map of algebras $\phi_{\Bbbk}: \U{\Bbbk} \twoheadrightarrow \Dist{G}$.  The elements in the kernel of $\phi_{\Bbbk}$ act as $0$ on any finite dimensional type \textbf{1} module for $\U{\Bbbk}$.  Such a module is therefore a finite dimensional $\Dist{G}$-module, and by \cite[II.1.20]{rags}, is also a rational $G$-module. 

A $\U{\mathcal{O}}$-module $M_{\mathcal{O}}$ will be called a $\U{\mathcal{O}}$-lattice if it is free of finite rank as an $\mathcal{O}$-module.  One obtains a $\U{\zeta}$-module
$$M_{\zeta} = M_{\mathcal{O}} \otimes_{\mathcal{O}} \mathbb{C}.$$
By the discussion above, if $M_{\mathcal{O}}$ is a type \textbf{1} module, then one obtains a $\Dist{G}$-module
$$M = M_{\mathcal{O}} \otimes_{\mathcal{O}} \Bbbk.$$

Let $\lambda \in \mathbb{X}^+$.  Let $V$ be a finite-dimensional $\U{\zeta}$-module of highest weight $\lambda$ that is generated by a weight vector $v_{\lambda}$.  Such a module will be a quotient of $\Delta_{\zeta}(\lambda)$, and the particular case of interest for us is when $V$ is $L_{\zeta}(\lambda)$.  One can always find a particular $\U{\mathcal{O}}$-lattice inside of $L_{\zeta}(\lambda)$ by taking the submodule $\U{\mathcal{O}}.v_{\lambda}$.  Such a construction is referred to as a minimal lattice, as it is necessarily contained in any other $\U{\mathcal{O}}$-lattice of $L_{\zeta}(\lambda)$.  By duality there also exists a maximal $\U{\mathcal{O}}$-lattice inside of $L_{\zeta}(\lambda)$.  The resulting $G$-modules obtained from the minimal and maximal lattices are denoted as
$$\Delta^{\textup{red}}(\lambda) \quad \text{and} \quad  \nabla^{\textup{red}}(\lambda)$$
respectively.  These are indecomposable modules for $G$, and both have formal characters equal to that of $L_{\zeta}(\lambda)$.  The symbols denoting each module point to their similarities with the standard and costandard $G$-modules of highest weight $\lambda$ (each of which can be constructed by minimal and maximal lattices respectively of finite dimensional $\mathfrak{g}_{\mathbb{C}}$-modules).  Specifically, we can place these modules in a chain of homomorphisms
$$\Delta(\lambda) \twoheadrightarrow \Delta^{\textup{red}}(\lambda) \twoheadrightarrow L(\lambda)$$
and
$$L(\lambda) \hookrightarrow \nabla^{\textup{red}}(\lambda) \hookrightarrow \nabla(\lambda).$$

The modules $L(\lambda)$ and $\Delta^{\textup{red}}(\lambda)$ have the same character if and only if
$$\Delta^{\textup{red}}(\lambda) \cong \nabla^{\textup{red}}(\lambda).$$

Another way to obtain a $\U{\mathcal{O}}$-lattice is to start with an indecomposable tilting module $T(\lambda)$ for $G$.  Andersen showed \cite[\S 4.2]{And2} that this tilting module can be lifted to an indecomposable tilting module $T_{\mathcal{O}}(\lambda)$ over $\Dist{G_{\mathcal{O}}}$ \cite[II.E.20]{rags}.  This pulls back to a type \textbf{1} tilting module for $\U{\mathcal{O}}$.  In this way, one obtains a tilting $\U{\zeta}$-module
$$T_{\mathcal{O}}(\lambda) \otimes_{\mathcal{O}} \mathbb{C}.$$
This module has the same character as $T(\lambda)$.  There are non-negative integers $n_{\lambda,\mu}$ such that
\begin{equation}\label{E:AlgToQuanTilt}
T_{\mathcal{O}}(\lambda) \otimes_{\mathcal{O}} \mathbb{C} \cong T_{\zeta}(\lambda) \oplus \bigoplus_{\mu < \lambda} T_{\zeta}(\mu)^{\oplus n_{\lambda,\mu}}.
\end{equation}

\subsection{}

The characters of finite dimensional $G$-modules and the characters of finite dimensional type \textbf{1} $\U{\zeta}$-modules are elements in the ring $\mathbb{Z}[X]^W$.  Define the left action of the commutative ring $\mathbb{Z}[X]^W$ on itself by
$$\eta.\sigma = \sigma \eta^F,$$
for all $\eta,\sigma \in \mathbb{Z}[X]^W$.  This action makes $\mathbb{Z}[X]^W$ into a free left $\mathbb{Z}[X]^W$-module, and following Donkin \cite{D2} we call a basis for this action a ``$p$-basis for $\mathbb{Z}[X]^W$."  One $p$-basis is given by the set of orbit sums
$$\{s(\lambda) \mid \lambda \in \mathbb{X}_p \}.$$
More generally we obtain a $p$-basis from any collection of elements
$$\{f(\lambda) \mid \lambda \in \mathbb{X}_p \},$$
where each $f(\lambda)$ is of the form
\begin{equation}\label{E:triangularity}
f(\lambda) = s(\lambda) + \sum_{\mu \in \mathbb{X}_p, \mu <_{\mathbb{Q}} \lambda} { s(\mu) g_{\mu}}^F, \quad g_{\mu} \in \mathbb{Z}[X]^W.
\end{equation}
In particular, the collections
$$\{t(\lambda)\}_{\lambda \in \mathbb{X}_p}, \quad \{q(\lambda)\}_{\lambda \in \mathbb{X}_p}, \quad \{t_{\zeta}(\lambda)\}_{\lambda \in \mathbb{X}_p},$$
each have this form and therefore are all $p$-bases.

\begin{lemma}\label{L:equaltop}
Let $\{\eta_{\lambda}\}_{\lambda \in \mathbb{X}_p}$ and $\{\sigma_{\lambda}\}_{\lambda \in \mathbb{X}_p}$ be collections of elements in $\mathbb{Z}[X]^W$.  If
$$\sum_{\lambda \in \mathbb{X}_p} q(\lambda)\eta_{\lambda}^F = \sum_{\lambda \in \mathbb{X}_p} t_{\zeta}(\lambda)\sigma_{\lambda}^F,$$
then $\eta_{(p-1)\rho}=\sigma_{(p-1)\rho}$.
\end{lemma}

\begin{proof}
We can take the sum
$$\sum_{\lambda \in \mathbb{X}_p} q(\lambda)\eta_{\lambda}^F$$
and express it in the $\{t_{\zeta}(\lambda)\}$-basis.  By Equation (\ref{E:triangularity}), and the fact that $(p-1)\rho$ is maximal among weights in $\mathbb{X}_p$ under the ordering $\le_{\mathbb{Q}}$, it follows that $\eta_{(p-1)\rho}^F$ will also be the coefficient of $t_{\zeta}((p-1)\rho)$ in the second expression.  The result now follows.
\end{proof}

\subsection{}
When the indecomposable projective $G_1$-modules lift to tilting modules for $G$, we have $q(\lambda)=t(\lambda)$, and when Lusztig's conjecture regarding the characters of the simple $G$-modules holds, we have $q(\lambda)=t_{\zeta}(\lambda)$.  So in very large characteristic, all three are equal.  In this subsection we will give general relationships between the quantum and algebraic quotients that hold in all situations.

First we establish a few lemmas.

\begin{lemma}
Let $\lambda \in \mathbb{X}_p$, $M$ in $G$-mod.  Then
$$\textup{ch}(\Hom_{G_1}(\widehat{Q}_1(\lambda),M))=\eta^F,$$
where $\eta^F$ is the coefficient of $q((p-1)\rho)$ when expressing $q((p-1)\rho-\lambda)\text{ch}(M)$ in the $\{q(\mu)\}$-basis.
\end{lemma}

\begin{proof}
There is an isomorphism of $T$-modules
$$\Hom_{G_1}(\widehat{Q}_1(\lambda),M) \cong \Hom_{G_1}(\Bbbk,\widehat{Q}_1(\lambda)^* \otimes M) .$$
The module $\widehat{Q}_1(\lambda)^* \otimes M$ is projective over $G_1T$, hence there is a decomposition of $G_1T$-modules
$$\widehat{Q}_1(\lambda)^* \otimes M \cong \bigoplus_{\mu \in \mathbb{X}_p} \widehat{Q}_1(\mu) \otimes \Hom_{G_1}(L(\mu),\widehat{Q}_1(\lambda)^* \otimes M)$$
The result now follows by taking the Steinberg quotients of the characters on each side of this decomposition (noting that $q((p-1)\rho-\lambda)\text{ch}(M)$ is the Steinberg quotient of $\widehat{Q}_1(\lambda)^* \otimes M$).
\end{proof}

For the quantum group $\U{\zeta}$, the projective modules for the small quantum group lift to tilting modules for all $p$.  Using an argument similar to the one just given, we obtain a similar result here.

\begin{lemma}
Let $\lambda \in (p-1)\rho+ X_+$, $M$ in $\U{\zeta}$-mod.  Then
$$\textup{ch}(\Hom_{\Usmall{\zeta}}(T_{\zeta}(\hat{\lambda}),M))=\eta^F$$
where $\eta^F$ is the coefficient of $t_{\zeta}((p-1)\rho)$ when expressing $t_{\zeta}((p-1)\rho-\lambda)\textup{ch}(M)$ in the $\{t_{\zeta}(\mu)\}$-basis.
\end{lemma}

We now have the following comparison theorem.

\begin{theorem}\label{T:comparison}
Let $\lambda \in \mathbb{X}_p(T)$.  Then
$$q(\lambda)= \sum_{\mu \in \mathbb{X}_p(T)} t_{\zeta}(\mu) \cdot \textup{ch}(\Hom_{G_1}(\widehat{Q}_1((p-1)\rho-\lambda),\nabla^{\textup{red}}((p-1)\rho-\mu))).$$
\end{theorem}

\begin{proof}
First, since $\{t(\mu)\}$ is a $p$-basis, there are coefficients $\eta_{\mu} \in \mathbb{Z}[X]^W$ such that
$$q(\lambda) = \sum_{\mu \in \mathbb{X}_p} t_{\zeta}(\mu){\eta_{\mu}}^F.$$
Fix some $\sigma \in \mathbb{X}_p$.  Since $\textup{ch}(\nabla^{\textup{red}}((p-1)\rho-\sigma))$ and $\textup{ch}(L_{\zeta}((p-1)\rho-\sigma))$ are equal, we have a character equality
$$q(\lambda)\textup{ch}(\nabla^{\textup{red}}((p-1)\rho-\sigma)) = \sum_{\mu \in \mathbb{X}_p} {t_{\zeta}}(\mu){\eta_{\mu}}^F\textup{ch}(L_{\zeta}((p-1)\rho-\sigma)).$$
Expanding the RHS into the $p$-basis $\{t_{\zeta}(\gamma)\}$, we find that the coefficient of $t_{\zeta}((p-1)\rho)$ is ${\eta_{\sigma}}^F$.
Expanding out the LHS into the $p$-basis $\{q(\gamma)\}$, we have that the coefficient of $q((p-1)\rho)$ is
$$\textup{ch}(\Hom_{G_1}(\widehat{Q}_1((p-1)\rho-\lambda),\nabla^{\textup{red}}((p-1)\rho-\sigma))).$$
The result now follows by Lemma \ref{L:equaltop}.
\end{proof}

Since
$$\Hom_{G_1}(\widehat{Q}_1((p-1)\rho-\lambda),\nabla^{\textup{red}}((p-1)\rho-\sigma)) \cong \Bbbk,$$
it follows that each $q(\lambda)$ is equal to $t_{\zeta}(\lambda)$ plus a non-negative sum of various $s(\mu)$.  The following then holds.

This now allows us to make a general statement about the coefficients $a_{\mu,\lambda}$.  Unlike the $b_{\mu,\lambda}$, we are not able to say in general that the nondecreasing pattern holds.  However, by comparing with the $c_{\mu,\lambda}$, we can say that all orbits that can appear, do.

\begin{cor}
If $p>2$, and $p>3$ if $G$ has a root system of type $\textup{G}_2$, then $a_{\mu,\lambda} \ge c_{\mu,\lambda}$ for all $\lambda \in \mathbb{X}_p$ and $\mu \le \lambda$.  In particular, if $\lambda \in \mathbb{X}_p$ and $\mu \in \mathbb{X}^+$ with $(\mu-\rho) \uparrow (\lambda - \rho)$, then $a_{\mu,\lambda} \ge 1$.
\end{cor}

\subsection{}
We also obtain useful facts about the characters $t(\lambda)$.  Andersen conjectured in \cite{And1} that for all $\lambda \in \mathbb{X}^+$ such that
$$\langle \lambda+\rho, \alpha_0^{\vee} \rangle < p^2,$$
it should hold that
$$T_{\mathcal{O}}(\lambda) \otimes_{\mathcal{O}} \mathbb{C} \cong T_{\zeta}(\lambda).$$
In other words, for such weights the algebraic and quantum tilting modules should agree.  When $p \ge 2h-2$, this conjecture implies Lusztig's conjecture, therefore it does not hold in general.

Nonetheless, it is an interesting question to consider.  As it pertains to Steinberg quotients, we note that
$$\langle (p-1)\rho+\lambda + \rho, \alpha_0^{\vee} \rangle = p(h-1)+\langle \lambda, \alpha_0^{\vee} \rangle.$$
Thus if
$$\langle \lambda, \alpha_0^{\vee} \rangle <  p(p - h + 1),$$
Andersen's conjecture would imply that
$$t(\lambda) = t_{\zeta}(\lambda).$$

It follows from (\ref{E:AlgToQuanTilt}) and Theorem \ref{T:OldMain} that for all $\lambda \in \mathbb{X}^+$, we have
$$t(\lambda)=t_{\zeta}(\lambda) + \sum_{(\mu-\rho) \uparrow (\lambda - \rho)} n^{\prime}_{\mu,\lambda} t_{\zeta}(\mu),$$
where $n^{\prime}_{\mu,\lambda}=n_{(p-1)\rho+\mu,(p-1)\rho+\lambda}$ from (\ref{E:AlgToQuanTilt}).
The next result follows directly from this observation.

\begin{prop}
Let $\lambda \in \mathbb{X}_p$.
\begin{enumerate}
    \item If for some $\mu$ we have $b_{\mu,\lambda}>c_{\mu,\lambda}$, then $b_{\gamma,\lambda}>c_{\gamma,\lambda}$ for all $(\gamma-\rho) \uparrow (\mu-\rho)$.
    \item Let $\gamma$ be the minimal dominant weight such that $(\gamma - \rho) \uparrow (\lambda-\rho)$.  Then $t(\lambda)=t_{\zeta}(\lambda)$ if and only if $b_{\gamma,\lambda}=c_{\gamma,\lambda}$.
\end{enumerate}
\end{prop}

\section{Steinberg quotients and Kazhdan-Lusztig polynomials}

\subsection{}

Let $\lambda \in \mathbb{X}_p$.  For all $\mu \in \mathbb{X}^+$ there are integers $d_{\mu,\lambda}$ such that

\begin{equation}\label{E:Weylsum}
    \text{ch}(L(\lambda)) = \sum d_{\mu,\lambda} \, \chi(\mu).
\end{equation}
We can compare these coefficients to the $a_{\mu,\lambda}$ thanks to a key result due to Kato \cite[Theorem 3.5]{K}.  Our proof follows Fiebig \cite{Fie}.

\begin{theorem}\label{T:WeylSteinberg}
Suppose that $\textup{ch}(L(\lambda-\rho))$ is given by Lusztig's character formula for all $\lambda \in \mathbb{X}_p$ such that $\lambda-\rho$ is $p$-regular.  Then for all $\lambda \in \mathbb{X}_p$ and $\mu \in \mathbb{X}^+$ we have
$$a_{\mu,\lambda} = |d_{\mu-\rho,\lambda-\rho}|.$$
\end{theorem}

\begin{proof}
Applying the translation principle, we may assume that $\lambda-\rho$ is strongly linked to $0$.  Therefore let $w \in W_p$ and $x \in W_p$ be such that
$$w \bigcdot 0 = \lambda-\rho, \quad x \bigcdot 0 = \mu-\rho.$$
Looking at both Corollary 3.4 and the proof of Theorem 3.5 from \cite{Fie}, we see that
$$[\widehat{Z}_1(w_0x \bigcdot 0): \widehat{L}_1(w_0w \bigcdot 0)] = |d_{x \, \bigcdot \, 0, w \, \bigcdot \, 0}|.$$
We then have
\begin{align*}
[\widehat{Z}_1(w_0x \bigcdot 0): \widehat{L}_1(w_0w \bigcdot 0)] & = [\widehat{Z}_1(p\rho + w_0x \bigcdot 0): \widehat{L}_1(p\rho + w_0w \bigcdot 0)]\\
& = [\widehat{Z}_1((p-1)\rho+(w_0x \bigcdot 0+\rho)): \widehat{L}_1((p-1)\rho+(w_0w \bigcdot 0 + \rho))]\\
& = (\widehat{Q}_1((p-1)\rho+(w_0w \bigcdot 0 + \rho)):\widehat{Z}_1((p-1)\rho+(w_0x \bigcdot 0+\rho))),
\end{align*}
The first equality above follows from the fact that for all $\mu,\sigma,\gamma \in X$ we have
$$[\widehat{Z}_1(\mu):\widehat{L}_1(\sigma)]=[\widehat{Z}_1(\mu+p\gamma):\widehat{L}_1(\sigma+p\gamma)].$$
We note that because the dot action of $W_p$ is a group action, we have for any $y \in W_p$ that $w_0y \bigcdot 0 = w_0 \bigcdot (y \bigcdot 0)$.  Thus the highest weight of
$$\widehat{Q}_1((p-1)\rho+(w_0w \bigcdot 0 + \rho))$$
is
\begin{align*}
    2(p-1)\rho+w_0((p-1)\rho+(w_0w \bigcdot 0 + \rho)) & = 2(p-1)\rho + w_0(p-1)\rho + w_0(w_0 \bigcdot w \bigcdot 0 + \rho)\\
    & = 2(p-1)\rho-(p-1)\rho+w_0(w_0(w \bigcdot 0 + \rho)-\rho + \rho)\\
    &= (p-1)\rho+w\bigcdot 0 + \rho\\
    &= (p-1)\rho+\lambda.
\end{align*}
We also have that the dominant weight in the $W$-orbit of $w_0x \bigcdot 0 + \rho$ is
\begin{align*}
    w_0(w_0 \bigcdot x \bigcdot 0 + \rho) & = w_0(w_0(x \bigcdot 0+\rho) - \rho + \rho)\\
    & = w_0(w_0(x \bigcdot 0) +\rho))\\
    & = x \bigcdot 0 + \rho\\
    & = \mu.
\end{align*}
This shows that $s(\mu)$ occurs in $q(\lambda)$ with multiplicity $|d_{\mu-\rho,\lambda-\rho}|$.
\end{proof}

\subsection{}

The result in the previous section was in terms of the algebraic group, but the same holds for the quantum group in view of the facts recalled in Section 4.  In this case, we indeed do have the hypothesis as holding, at least when $p>h$ (see \cite[II.H.12]{rags}).

We further note that in the case of the quantum group, if $\lambda-\rho$ is $p$-regular, then the orbit multiplicities in $t_{\zeta}(\lambda)$ are given by evaluations of Kazhdan-Lusztig polynomials for all $\lambda \in \mathbb{X}^+$.  This follows from Proposition \ref{P:quantumtiltingtensor} and \cite[Corollary 4.10]{K}.


\section{Lower bounds on Steinberg Quotients}

\subsection{} In this section we consider the primary problem of computing $t(\lambda)$ and $t_{\zeta}(\lambda)$ in some reasonable way.  Specifically, we seek to find properties that completely determine these characters.  Using Weyl characters as a guide, we will define an approximation for $t_{\zeta}(\lambda)$ by keying in on one of its properties.  These approximations will be lower bounds on Steinberg quotients, and are computable by a straightforward algorithm.  We will then provide computational evidence that suggests that these approximations are reasonably close to the $t_{\zeta}(\lambda)$ (or at least are so in some nontrivial examples).

\subsection{}

Donkin gave alternate proof of Weyl's formula in \cite{D4}, also recounted in \cite[Proposition II.5.10]{rags}.  Boiling it down to its essence, we see that Weyl characters are a subset of Euler characteristics, and  Euler characteristics can be shown to satisfy the following properties:

\begin{itemize}
    \item For every $\lambda \in \mathbb{X}^+$ and $w \in W$,
    $$\chi(w \bigcdot \lambda)=(-1)^{\ell(w)}\chi(\lambda).$$
    \item For all $\lambda \in \mathbb{X}^+$,
    $$s(\lambda)=\sum_{w \in W/W_{\lambda}} \chi(w\lambda).$$
    \item For all $\lambda \in \mathbb{X}$,
    $$\chi(\lambda) \in \mathbb{Z}[\mathbb{X}]^W.$$
\end{itemize}

Once we know that for $\lambda \in \mathbb{X}^+$, the highest weight of $\chi(\lambda)$ is $\lambda$, occuring once, then these properties completely determine $\chi(\lambda)$.  Together they provide a recursive way to compute the orbit sums that appear in $\chi(\lambda)$, starting with the outer orbit.  This recursive process can be expressed succinctly as a division between signed $W$-orbit sums in $\mathbb{Z}[\mathbb{X}]$.

\subsection{}

For each $\lambda \in \mathbb{X}^+$, let $\mathcal{I}(\lambda)$ denote the ``smallest'' element in $\mathbb{Z}[\mathbb{X}]^W$ of highest weight $\lambda$ having the property that, for all $\mu \in \mathbb{X}^+$,
$$\mathcal{I}(\lambda)\chi(\mu)$$
has the character of a good filtration module.  It is then trivially true that $\mathcal{I}(\lambda)=\chi(\lambda)$.  This is because $\mathcal{I}(\lambda)\chi(0)=\mathcal{I}(\lambda)$, so if $\mathcal{I}(\lambda)$ satisfies the condition above then $\mathcal{I}(\lambda)$ is itself a nonnegative sum of Weyl characters which has highest weight $\lambda$.  Thus $\chi(\lambda)$ appears at least once in this sum.  Since $\mathcal{I}(\lambda)$ is meant to be minimal with respect to this property, we must have that $\mathcal{I}(\lambda)=\chi(\lambda)$.

Motivated by this observation, and looking at Theorem \ref{T:MainGen}, we make the following definition.  We say that an element $\eta \in \mathbb{Z}[\mathbb{X}]^W$ has a \textit{good Steinberg multiplication} if
\begin{equation}\label{E:goodfiltrationcharacter}
\eta\chi((p-1)\rho+p\gamma) \quad \text{is a good filtration character} \quad \forall \gamma \in \mathbb{X}^+.
\end{equation}
We then define, for each $\lambda \in \mathbb{X}^+$, the character $\mathcal{M}_p(\lambda)$ to be the smallest element in $\mathbb{Z}[\mathbb{X}]^W$ having highest weight $\lambda$ and satisfying the good Steinberg multiplication property.

We will make precise what we mean by ``smallest'' by way of an algorithm for computing $\mathcal{M}_p(\lambda)$ (this algorithm will evidently make a minimal choice at each stage).  We will then prove that this algorithm always produces a well-defined element that has the good Steinberg multiplication property.

Before proceeding, we show that the good Steinberg multiplcation property, which a priori involves checking an infinite number of character multiplications, can in fact be checked by multiplying by a finite number of characters.  At the same time, it turns out to be too optimistic to hope that this property can simply be checked by multiplication against the Steinberg character itself.

\begin{theorem}\label{T:Thisisenough}
Let $\eta \in \mathbb{Z}[\mathbb{X}]^W$.  Let $m$ be an integer such that $pm > \langle \sigma, \alpha_0^{\vee} \rangle$ as $\sigma$ ranges over the weights in $\eta$.  Then the following hold.

\begin{enumerate}
\item $\eta$ satisfies the good Steinberg multiplication property if and only if
$$\eta\chi((p-1)\rho+p\gamma) \quad \text{is a good filtration character} \quad \forall \gamma \in \mathbb{X}_{m}.$$
\item If $\eta$ is equal to the character of a $G$-module, then $\eta$ satisfies the good Steinberg multiplication property if and only if
$$\eta\chi((p-1)\rho) \quad \text{is a good filtration character}.$$
\item If $\eta$ is not the character of a $G$-module, then in general the previous condition fails.
\end{enumerate}

\end{theorem}

\begin{proof}
(1) The only thing to prove is that it suffices to check the property on this finite set.  To do so, we show that for any $\gamma \not\in  \mathbb{X}_{m}$, there exists an element $\gamma_{m} \in \mathbb{X}_{m}$ such that $\eta\chi((p-1)\rho+p\gamma)$ is a good filtration character if and only if $\eta\chi((p-1)\rho+p\gamma^{\prime})$ is a good filtration character.

Let $J \subset \Pi$ be the set of all simple roots $\alpha_i$ such that
$$\langle \gamma,\alpha_i^{\vee} \rangle < m.$$
Define
$$\gamma_1 = \sum_{\alpha_i \in J} \langle \gamma, \alpha_i^{\vee} \rangle \varpi_i.$$
That is, in terms of the basis of fundamental dominant weights, $\gamma_1$ agrees with $\gamma$ on those coefficients that are less than $m$, and is $0$ on the others.  Now define
$$\gamma_2 = \sum_{\alpha_i \in \Pi \backslash J} m\varpi_i,$$
and set
$$\gamma^{\prime}=\gamma_1+\gamma_2.$$
We have that
$$\gamma=\gamma_1 + (\gamma-\gamma_1),$$
and the coefficients of $\gamma-\gamma_1$ are greater than or equal to those of $\gamma_2$.  Suppose now that $\sigma$ is a weight in $\eta$.  Then by Lemma \ref{L:NeededReflections} there is some $w \in W_J$ such that
$$w \bigcdot ((p-1)\rho + p\gamma + \sigma)$$
is in $\mathbb{X}^+ -\rho$.  Note that  $w(\gamma-\gamma_1)=\gamma-\gamma_1$ for any such $w$, as this weight pairs to $0$ with all $\alpha_i^{\vee}$ for $\alpha_i \in J$.  Applying Lemma \ref{L:dotandregularaction}, we then have
\begin{align*}
w \bigcdot ((p-1)\rho + p\gamma + \sigma) & = w \bigcdot ((p-1)\rho + p\gamma_1 + \sigma) + w(p(\gamma-\gamma_1))\\
& = w \bigcdot ((p-1)\rho + p\gamma_1 + \sigma) + p(\gamma-\gamma_1).
\end{align*}
The reasoning above also implies for this same $w$ that the weight
\begin{align*}
w \bigcdot ((p-1)\rho + p\gamma^{\prime} + \sigma) & = w \bigcdot ((p-1)\rho + p\gamma_1 + \sigma) + w(p\gamma_2)\\
& = w \bigcdot ((p-1)\rho + p\gamma_1 + \sigma) + p\gamma_2
\end{align*}
is in $\mathbb{X}^+ - \rho$.  It therefore follows that in the basis of Weyl characters, $\eta\chi((p-1)\rho+p\gamma)$ equals
$$\sum_{\mu \in \mathbb{X}} c_{\mu}\chi(\mu+p(\gamma-\gamma_1)),$$
if and only if in this basis $\eta\chi((p-1)\rho+p\gamma^{\prime})$ equals
$$\sum_{\mu \in \mathbb{X}} c_{\mu}\chi(\mu+p\gamma_2).$$
(Note that $\mu$ need not be dominant in these expressions, but must be after adding $p(\gamma-\gamma_1)$ or $p\gamma_2$.)  This finishes the proof of (1).

(2) This is shown in Section 2 of \cite{And3}.  

(3) A counter-example is given in Section \ref{SS:A3}.
\end{proof}

We can now present our algorithm for determining $\mathcal{M}_p(\lambda)$.

\begin{algorithm}\label{Algorithm1}
\textbf{Input:} a weight $\lambda \in \mathbb{X}^+$.

\vspace{0.1in}
\noindent \textbf{Output:} the character $\mathcal{M}_p(\lambda)$.

\vspace{0.1in}
\noindent \textbf{Preliminary Steps:}
\begin{enumerate}
\item Set $\Psi^+(\mathcal{M}_p(\lambda))$ to be the set of all $\gamma \in \mathbb{X}^+$ such that $(\gamma-\rho) \uparrow (\lambda-\rho)$.
\vspace{0.05in}

\item Let $m$ be the minimal integer such that $\langle \lambda, \alpha_0^{\vee} \rangle < pm$.  Recall that $\mathbb{X}_m \subseteq \mathbb{X}^+$ denotes the set of $m$-restricted weights.
\vspace{0.05in}

\item Initialize $\mathcal{M}_p(\lambda)=s(\lambda)$.
\end{enumerate}

\vspace{0.1in}
\noindent \textbf{Iterative Steps:}
\begin{enumerate}
\item Let $\mu \in \Psi^+(\mathcal{M}_p(\lambda))$ be a maximal weight, under the $\le$ ordering, such that the multiplicity of $s(\mu)$ in $\mathcal{M}_p(\lambda)$ is $0$.  If no such $\mu$ exists, then the process is ended.  Otherwise, proceed to Step (2).

\vspace{0.05in}
\item For each $\gamma \in \mathbb{X}_m$, write the character
$$\chi((p-1)\rho+p\gamma)\mathcal{M}_p(\lambda)$$
in the basis of Weyl characters.

\vspace{0.05in}
\item Let $x_{\mu}$ be the least integer appearing on as a coefficient on a term of the form
$$\chi((p-1)\rho+p\gamma + w\mu)$$
in the previous step, for all $w \in W$ and for all $\gamma \in \mathbb{X}_m$.

\vspace{0.05in}
\item Add $-x_{\mu}s(\mu)$ to $\mathcal{M}_p(\lambda)$ and repeat Step (1).
\end{enumerate}

\end{algorithm}

\noindent \textbf{Claim:} This algorithm terminates after a finite number of computations and returns the same element regardless of any choices made at any stage.  The element it returns satisfies the good Steinberg multiplication property.

\begin{proof}
Regarding the termination of the algorithm, because the set $\mathbb{X}_m$ is finite, it is clear that the computations made at every stage are finite.  Also, after each loop the coefficient on the relevant $s(\mu)$ becomes positive, so the loop never returns to a previously settled case.  To see this, suppose that $\mu^{\prime}$ is minimal such that $(\mu-\rho) \uparrow (\mu^{\prime}-\rho)$.  If $\mu$ is (at this stage) a maximal weight such that the coefficient of $s(\mu)$ is $0$, then it must be true that $s(\mu^{\prime})>0$.  Now, applying the argument in the proof of Theorem \ref{T:MainGen}(2), it will follow that after this loop the coefficient of $s(\mu)$ is at least as big as the coefficient of $s(\mu^{\prime})$.

Next, we will show that at Step (1) in the iterative process, if there are two or more maximal weights satisfying the condition (necessarily incomparable to each other), the the order in which they are chosen does not matter.  Suppose then that $\mu_1$ and $\mu_2$ are two such weights, and that when adding in the orbit sums $s(\mu_1)$, we impact the number of $s(\mu_2)$ needed later.  This would mean that there are elements $w_1, w_2 \in W$ such that $(p-1)\rho+p\gamma+w_2\mu_2 \in \mathbb{X}^+-\rho$, and
$$\chi((p-1)\rho+p\gamma+w_1\mu_1)=\pm \chi((p-1)\rho+p\gamma+w_2\mu_2).$$
If
$$(p-1)\rho+p\gamma+w_1\mu_1=(p-1)\rho+p\gamma+w_2\mu_2,$$
then we would have $w_1\mu_1=w_2\mu_2$ for $w_1$ and $w_2$ elements in the finite Weyl group.  But this cannot happen since $\mu_1$ and $\mu_2$ are distinct dominant weights.  The only other case then is that for some nontrivial $w \in W$,
$$ w \bigcdot((p-1)\rho+p\gamma+w_1\mu_1)=(p-1)\rho+p\gamma+w_2\mu_2.$$
However, applying Lemma \ref{L:StraighteningInterior}, it would follow that $w_2\mu_2$ lies in the convex hull of $W(w\mu_1)$, so that $\mu_2$ lies in the convex hull of $W\mu_1$.  But this is well known to imply that $\mu_2 \le \mu_1$, which contradicts our assumption that these weights are not comparable.

Finally, by Theorem \ref{T:Thisisenough} we know that the final value for $\mathcal{M}_p(\lambda)$ has the good Steinberg multiplication property.
\end{proof}

\subsection{}

We wish to compare the multiplicities between $t_{\zeta}(\lambda)$ and $\mathcal{M}_p(\lambda)$.  As some explicit computations indicate, these characters will in general differ, so that $t_{\zeta}(\lambda)$ is defined by more than just having the good Steinberg multiplication property.  Nonethelss, in a very special case we do always have equality.

\begin{prop}\label{P:extremecase}
For all $\lambda \in \mathbb{X}^+$,
$$\mathcal{M}_p(p\lambda)=\chi(\lambda)^F.$$
In particular, $\mathcal{M}_p(p\lambda)=t_{\zeta}(p\lambda)$.
\end{prop}

\begin{proof}
The only orbit sums that are needed to appear in $\mathcal{M}_p(p\lambda)$ are those $s(\mu)$ such that
$$(\mu-\rho) \uparrow (p\lambda-\rho).$$
But such $\mu$ are necessarily in $p\mathbb{X}$.  It follows that there are distinct dominants weights $\mu_1, \mu_2, \ldots, \mu_m$ which are less than $\lambda$, and coefficients $c_i \in \mathbb{Z}$, such that
$$\mathcal{M}_p(p\lambda)=\chi(\lambda)^F + \sum c_i\chi(\mu_i)^F.$$
From the definition of $\mathcal{M}_p(p\lambda)$, it is necessary that
\begin{align*}
\chi((p-1)\rho)\mathcal{M}_p(p\lambda) & = \chi((p-1)\rho)\left(\chi(\lambda)^F + \sum c_i\chi(\mu_i)^F\right)\\
& = \chi((p-1)\rho+p\lambda) + \sum c_i\chi((p-1)\rho+p\mu_i)\\
\end{align*}
is a good filtration character.  The Weyl characters appearing in this last expression are all distinct, therefore the coefficients $c_i$ are nonnegative.  Since $\chi(\lambda)^F$ itself has the good Steinberg multiplication property, it follows by the minimality of $\mathcal{M}_p(\lambda)$ that all $c_i=0$, so that $\mathcal{M}_p(p\lambda)=\chi(\lambda)^F$.

The last part follows from Proposition \ref{P:quantumtiltingtensor}.
\end{proof}


\section{Computations in Type $A$}

\subsection{}

For $A_n$, the fundamental dominant weights are $\varpi_1, \varpi_2, \ldots \varpi_n$.  We write
$$(a_1,a_2,\ldots,a_n)=a_1\varpi_1+a_2\varpi_2+ \cdots + a_n\varpi_n.$$

\subsection{$A_3$}\label{SS:A3}

We will work here explicitly with $p=5$.  Lusztig's conjecture is known to hold for the algebraic group for all $p \ge 5$, and a detailed listing of the characters can be found in \cite[II.8.20]{rags}.  We also know, by \cite{BNPS2}, that $q(\lambda)=t(\lambda)$ for all $\lambda \in \mathbb{X}_p$ for all $p$.

Applying the general formula given in \cite[II.8.20]{rags}, we have that

$$L(3,2,3)=\chi(3,2,3) -\chi(2,2,2) - \chi(5,0,1) - \chi(1,0,5) + \chi(1,1,3) + \chi(3,1,1) - 2\chi(2,0,2)+3\chi(0,0,0).$$

Because Lusztig's theorem holds here in both cases, we are able by Theorem \ref{T:WeylSteinberg} to immediately read off the (equal) characters $q(4,4,4)$ and $t_{\zeta}(4,4,4)$.  By the comments above, these also equal $t(4,4,4)$.  We have

$$t(4,4,4)=s(4,4,4) + s(2,4,2) + s(7,1,1) + s(1,1,7) + s(3,3,1) + s(1,3,3) + 2s(2,2,2)+3s(1,2,1).$$

A computer calculation further reveals that these are the minimal orbit multiplicities needed to satisfy the good Steinberg multiplication property, so that this character is also equal to $\mathcal{M}_p(4,4,4)$.  At the same time, our algorithm found that
$$\chi((p-1)\rho)(t(4,4,4) - s(2,4,2))$$
is a good filtration character.  Thus the orbit $s(2,4,2)$ is not needed to obtain a good filtration character when multiplying by the Steinberg weight.  This gives an example proving claim (3) in \ref{T:Thisisenough}.

\subsection{$A_4$}\label{SS:A4}

Computer calculations by Scott \cite{Sc} have shown that Lusztig's conjecture holds for the algebraic group for $p=5, 7$.  Applying \cite{BNPS4}, we also have for $p=7$ that $q(\lambda)=t(\lambda)$.   We will therefore look at $p=7$, where it follows then that $t_{\zeta}(\lambda)=t(\lambda)$.  There are $52$ dominant weights appearing in $t(6,6,6,6)$ (this corresponds to the number of dominant alcoves that are less than the top restricted alcove under the $\uparrow$ ordering).  For brevity, we list only the orbit sums appearing with the greatest multiplicities, which are $9$, $13$, and $21$.
$$$$

\begin{tabular}{c|c|c}
Orbit Sum & Multiplicity in $t_{\zeta}(6,6,6,6)$ & Multiplicity in $\mathcal{M}_p(6,6,6,6)$\\
\hline
$s(4, 2, 2, 4)$  & $9$ & $9$\\
$s(1, 5, 5, 1)$ & $9$ & $9$\\
$s(3, 1, 2, 6)$ & $9$ & $9$\\
$s(6, 2, 1, 3)$ & $9$ & $9$\\
$s(5, 1, 2, 3)$ & $13$ & $13$\\
$s(3, 2, 1, 5)$ & $13$ & $13$\\
$s(4, 1, 1, 4)$ & $21$ & $20$\\
$s(1, 1, 1, 1)$ & $21$ & $20$\\
\end{tabular}

\vspace{0.2in}
From this computation we find the interesting fact that the $\mathcal{M}_p(\lambda)$ are not always equal to the $t_{\zeta}(\lambda)$.

\subsection{$A_5$}

Here we do not know of any small primes in which Lusztig's conjecture holds for the algebraic group, but we do know that it holds for $p > 6$ in the quantum setting.  Using the Kazhdan-Lusztig polynomials computed by Frank L\"{u}beck, all the computations that we were able to make showed complete agreement between $t_{\zeta}(\lambda)$ and $\mathcal{M}_p(\lambda)$.  The largest example for which we were able to compute $\mathcal{M}_p(\lambda)$ was for the weight $(3,6,2,4,4)$.

We found that the characters $t_{\zeta}(3,6,2,4,4)$ and $\mathcal{M}_p(3,6,2,4,4)$ were equal.  In these characters there are $79$ different orbit sums appearing.  The lowest orbit sum, $s(1,1,1,1,1)$, occurs with the greatest multiplicity, which is $23$.

\subsection{Larger ranks}

L\"{u}beck shared complete computations of the relevant Kazhdan-Lusztig polynomials for type $A$ up through rank $7$.  Unfortunately, our own code for computing the characters $\mathcal{M}_p(\lambda)$ has not been able to keep up!  Our first attempt was an ad hoc script written in MATLAB that replaced by-hand computations.  Subsequent versions have been modifications of this, but a more serious design is needed (we did not expect the $\mathcal{M}_p(\lambda)$ to be as close to the $t_{\zeta}(\lambda)$ as we have found them to be).  Ideally we would like to find more compact formulas for computing the characters $\mathcal{M}_p(\lambda)$, possibly one that is patterned after Freudenthal's formula for computing $\chi(\lambda)$.

We should also say at this point that the data grows dramatically we move up in rank, so it will be illuminating to study the characters side-by-side for larger $n$.  For example, in type $A_5$, $p=7$, the Steinberg quotient $t(6,6,6,6,6)$ contains $478$ distinct orbits, and the biggest orbit multiplicity is $646$.  Moving into larger ranks, and stating everything now in terms of a general prime $p>h$, we find for type $A_6$ that the Steinberg quotient $t((p-1)\rho)$ has $5706$ distinct orbits, with the greatest orbit sum multiplicity being $65199$.  For $A_7$, the corresponding character has $83824$ distinct orbits, and the largest multiplicity is more than $34$ million.

These numbers all come from the Kazhdan-Lusztig polynomial computations made by L\"{u}beck, though we had independently computed the numbers of distinct orbits in the characters just listed.

\section{Concluding Remarks}

We find the characters $\mathcal{M}_p(\lambda)$ to be quite interesting.  Although they are not in complete agreement with the $t_{\zeta}(\lambda)$, which when $p>h$ are computable by evaluating Kazhdan-Lusztig polynomials if $\lambda-\rho$ is $p$-regular, they are nonetheless close in the examples we have seen.  Finding further insights into this relationship will be the focus of future work.

It is worth observing also that the algorithm for computing $\mathcal{M}_p(\lambda)$ does not require $\lambda-\rho$ to be $p$-regular.  In fact, intuitively it seems reasonable that the more $p$-singular the weight $\lambda-\rho$ is, the closer $\mathcal{M}_p(\lambda)$ will be to $t_{\zeta}(\lambda)$ (in terms of relative difference).  See Proposition \ref{P:extremecase}.

In a similar vein, some calculations in the case of $A_4$ indicate that the behavior under translation is not the same for $\mathcal{M}_p(\lambda)$ as it is for $t_{\zeta}(\lambda)$, and this likely accounts for the descrepency detailed in Section \ref{SS:A4}.  This will also be explored further.


\providecommand{\bysame}{\leavevmode\hbox
to3em{\hrulefill}\thinspace}

\end{document}